\RequirePackage{fix-cm}
\documentclass[smallextended]{svjour3}       
\smartqed  
\usepackage{graphicx}

 \usepackage{amsmath}
\usepackage{epstopdf} 
 \usepackage{algorithm}
\usepackage{algorithmic}
\usepackage{mathptmx}      
\usepackage{mathrsfs}
\usepackage{color}
\usepackage{amssymb}
\usepackage{cite}
\usepackage{subfigure}
 \usepackage{multirow}
\usepackage{booktabs,longtable}
 \usepackage{booktabs}
\usepackage{threeparttable}

\begin{document}

\title{Block sampling Kaczmarz-Motzkin methods for consistent linear systems\thanks{This work was funded by the National Natural Science Foundation of China (No. 11671060) and the Natural Science Foundation Project of CQ CSTC (No. cstc2019jcyj-msxmX0267).}
}


\author{Yanjun Zhang  \and Hanyu Li    }


\institute{  \at
              College of Mathematics and Statistics, Chongqing University, Chongqing 401331, P.R. China\\
              \email{yjzhang@cqu.edu.cn; \  lihy.hy@gmail.com or hyli@cqu.edu.cn}           
           \and
           }

\date{Received: date / Accepted: date}

\maketitle

\begin{abstract}
The sampling Kaczmarz-Motzkin (SKM) method is a generalization of the randomized Kaczmarz and Motzkin methods. It first samples
some rows of coefficient matrix randomly to build a set and then makes use of the maximum violation criterion within this
set to determine a constraint. Finally, it makes progress by enforcing this single constraint.
 In this paper, on the basis of the framework of the SKM method and considering the greedy strategies,
we present two block sampling Kaczmarz-Motzkin methods for consistent linear systems. Specifically, we also first sample a subset of
rows of coefficient matrix and then determine an index in this set using the maximum violation criterion. Unlike the SKM method, in the rest of the block methods, we devise different greedy strategies to build index sets.
Then, the new methods make progress by enforcing the corresponding multiple constraints simultaneously.
Theoretical analyses demonstrate that these block methods converge at least as quickly as the SKM method, and numerical experiments show that, for the same accuracy, our methods outperform the SKM method in terms of the number of iterations and computing time.
\end{abstract}

\section{Introduction}
We consider the following consistent linear systems
\begin{equation}
\label{1}
Ax=b,
\end{equation}
where $A\in R^{m\times n}$ with $m\gg n $, $b\in R^{m}$, and $x$ is the $n$-dimensional unknown vector. As we know, the Kaczmarz method \cite{kaczmarz1} is a popular so-called row-action method for solving the systems (\ref{1}). In 2009, Strohmer and Vershynin \cite{Strohmer2009} proved the linear convergence of the randomized Kaczmarz (RK) method. 
Subsequently, many randomized Kaczmarz type methods were proposed for different possible systems settings; see for example \cite{Needell2010, Eldar2011, Completion2013, Completion2015, Dukui2019, Wu2020, Chen2020} and references therein.

Unlike the RK method which selects the working rows of $A$ according to some probability distribution, 
Motzkin method \cite{Agamon54,Motzkin54} 
employs a greedy strategy, i.e., the maximum violation criterion, to select the working row in each iteration. So, the method is also known as the Kaczmarz method with the ``most violated constraint control'' or ``maximal-residual control'' \cite{Petra2016,Nutini2016,Nutini2018}.  This greedy strategy makes the Motzkin method outperform the RK method in many cases. 
So, many analyses and applications of the Motzkin method were published recently; see for example \cite{haddock2019motzkin,rebrova2019sketching,li2020novel} and references therein.

The sampling Kaczmarz-Motzkin (SKM) method proposed in \cite{de2017sampling}  for solving linear feasibility problem is a combination of the RK and Motzkin methods. 
Its accelerated version 
was presented in \cite{morshed2020accelerated}, which introduces  Nesterov's acceleration scheme. In addition, an improved analysis of the SKM method was given in \cite{haddock2019greed}. 
The SKM method overcomes some drawbacks of the methods of RK and Motzkin. 
For example, the Motzkin method is expensive since it selects the index $i_k$ at the iteration $k$ by comparing the residual errors of all the constraints, and 
the RK method may make progress slowly since it doesn't employ a greedy strategy. 
Instead, the SKM method can select an index by just comparing the residual errors of part of constraints, and employs the maximum violation criterion. 
However,  the SKM method may be slow because it makes progress by enforcing only one constraint.  
Inspired by the block algorithms given in \cite{needell2014paved,needell2015randomized,Niu2020} which can accelerate the original ones, in this paper, we consider the block versions of the SKM method. 
Two greedy strategies are devised to determine the index sets for block iteration.  


The rest of this paper is organized as follows. In Section \ref{sec2}, some notation and preliminaries are given. Our methods 
and their convergences are discussed in Section \ref{sec3} and Section \ref{sec4}, respectively. Finally, we present the numerical results in Section \ref{sec5}.

\section{Notation and preliminaries}\label{sec2}
Throughout the paper, for a matrix $A$, $A_{(i)}$ and ${\rm R(A)}$ denote its $i$-th row (or $i$-th entry in the case of a vector) and column space, respectively. We denote the number of elements of a set $\mathcal{I}$ by $|\mathcal{I}|$ and let the positive eigenvalues of $A^{T}A$, where $(\cdot)^{T}$ denotes the transpose of a vector or a matrix, be always arranged in algebraically nonincreasing order:
$$\lambda_{\max }=\lambda_{1} \geq \lambda_{2}  \geq \cdots  \geq \lambda_{\min }>0.$$
To analyze the convergence of our new methods, the following fact will be used extensively later in the paper.
\begin{lemma}
\label{theorem1}(\cite{horn2012matrix})
Let $A \in R^{n\times n}$ be symmetric and $A_t \in R^{t\times t}$ be its principal submatrix. Then
$$\lambda_{n-t+i}(A) \leqslant \lambda_{i}\left(A_{t}\right) \leqslant \lambda_{i}(A), \quad i=1, \cdots, t.$$
\end{lemma}

In addition, to compare the SKM method and our new methods clearly, we list the SKM method from \cite{de2017sampling} in Algorithm \ref{alg1}.

\begin{algorithm}
\caption{ The SKM method} 
\label{alg1}
\begin{algorithmic}
\STATE{\mbox{Input:} ~ Matrix $A\in R^{m\times n}$, vector $b\in R^{m}$, parameter $\beta$, initial estimate $x_0$.}
\STATE{\mbox{Output:} ~Approximate $x$ solving $Ax=b$.}
\STATE{1. Choose a sample of $\beta$ constraints, $\tau_k$, uniformly at random from among the rows of $A$.}
\STATE{2. Set $t_k={\rm arg} \max \limits _{i\in \tau_k} ( b_{(i)}- A_{(i)} x_{k})^2  $.}
\STATE{3. Update
\begin{align}
  x_{k+1}=x_{k}+\frac{ b_{(t_{k})}-A_{(t_{k})}x_k}{ \| A_{\left(t_{k}\right)} \|_{2}^{2}}( A_{(t_{k})})^{T}.\notag
\end{align}}
\STATE{4. Repeat.}
\end{algorithmic}
\end{algorithm}

\section{The first block sampling Kaczmarz-Motzkin method }\label{sec3}
The first block sampling Kaczmarz-Motzkin (BSKM1) method is presented in Algorithm \ref{alg2}.  Compared with the SKM method, the main difference and key of the BSKM1 method is to devise the index set for updating the approximation. 
We mainly use 
the threshold value $\delta_k$ obtained from the small set $\tau_k$ which is the same as the one in Algorithm \ref{alg1} to build the index set. 



\begin{algorithm}
\caption{ The BSKM1 method}
\label{alg2}
\begin{algorithmic}
\STATE{\mbox{Input:} ~ Matrix $A\in R^{m\times n}$, vector $b\in R^{m}$, parameter $\beta$, initial estimate $x_0$.}
\STATE{\mbox{Output:} ~Approximate $x$ solving $Ax=b$.}
\STATE{1. Choose a sample of $\beta$ constraints, $\tau_k$, uniformly at random from among the rows of $A$.}
\STATE{2. Set $t_k={\rm arg} \max \limits _{i\in \tau_k} ( b_{(i)}- A_{(i)} x_{k})^2  $, and $\delta_k=\max \limits _{i\in \tau_k} ( b_{(i)}- A_{(i)} x_{k})^2$.}
\STATE{3. Determine the index set
\begin{align}
  \mathcal{I}_{k}=\{h_k|( b_{(h_k)}- A_{(h_k)} x_{k})^2\geq \delta_k;~h_k\in[m]/\tau_k\}\cup\{t_k\}.\notag
\end{align}}
\STATE{4. Update
\begin{align}
  x_{k+1}=x_{k}+A_{\mathcal{I}_{k}}^{\dagger}(b_{\mathcal{I}_{k}}-A_{\mathcal{I}_{k}}x_{k}).\notag
\end{align}}
\STATE{5. Repeat.}
\end{algorithmic}
\end{algorithm}

\begin{remark}

\label{rmkgl}
Note that if
$$( b_{(i_k)}- A_{(i_k)} x_{k})^2 =  \max \limits _{1 \leq i  \leq m}( b_{(i)}- A_{(i)} x_{k})^2 ,$$
then $i_k\in\mathcal{I}_{k}.$ So the index set $\mathcal{I}_{k}$ in Algorithm \ref{alg2} is always nonempty. 
\end{remark}
\begin{remark}
\label{rmkg2}
Compared with the SKM method, in each iteration,  the BSKM1 method can eliminate several large violated constraints control simultaneously. So, the BSKM1 method converges faster; see the detailed discussions following Theorem \ref{theorem2}. 
 In addition, if $\beta=m$, the BSKM1 method reduces to the Motzkin method. This means that the BSKM1 method can be also regarded as the block version of the  Motzkin method.
\end{remark}

Now, we provide the 
convergence of Algorithm \ref{alg2}.

\begin{theorem}
\label{theorem2}
From an initial guess $x_0\in{\rm R(A^T)}$, the sequence $\{x_k\}_{k=0}^\infty$ generated by the BSKM1 method converges linearly in expectation to the least-Euclidean-norm solution $x_{\star}=A^{\dag}b$ and
\begin{eqnarray*}
\textrm{E}\|x_{k+1}-x_\star\|^{2}_2\leq\left(1- \frac{\beta}{\xi_k}\frac{ |\mathcal{I}_{k}|}{ m}\frac{\lambda_{\min}(A^TA)}{\lambda_{\max} (A_{\mathcal{I}_{k}}^TA_{\mathcal{I}_{k}})}\right)\| x_{k}-x_\star \|^2_2,
\end{eqnarray*}
where
\begin{eqnarray}\label{1000}
\xi_k=\frac{\sum  \limits _{\tau_k \in \binom{m}{\beta}} \left \|A_{\tau_k} x_{k}- b_{\tau_k}\right\|_{2}^{2}}{\sum  \limits _{\tau_k \in \binom{m}{\beta}}  \left\|A_{\tau_k} x_{k}- b_{\tau_k}\right\|_{\infty}^{2}}.
\end{eqnarray}
\end{theorem}

\begin{proof}
From Algorithm \ref{alg2}, using the fact $Ax_\star=b$, we have
\begin{eqnarray*}
x_{k+1}-x_\star&=&x_{k}-x_\star+A_{\mathcal{I}_{k}}^{\dagger}(b_{\mathcal{I}_{k}}-A_{\mathcal{I}_{k}}x_{k})
\\
&=& x_{k}-x_\star-A_{\mathcal{I}_{k}}^{\dagger}A_{\mathcal{I}_{k}}(x_{k}-x_{\star})
\\
&=&  (I- A_{\mathcal{I}_{k}}^{\dagger}A_{\mathcal{I}_{k}} )( x_{k}-x_\star).
\end{eqnarray*}
Since $A_{\mathcal{I}_{k}}^{\dagger}A_{\mathcal{I}_{k}}$ is an orthogonal projector, taking the square of the Euclidean norm on both sides and using the Pythagorean theorem, we get
\begin{align}
\|x_{k+1}-x_\star\|^{2}_2
&=   \|(I- A_{\mathcal{I}_{k}}^{\dagger}A_{\mathcal{I}_{k}} )( x_{k}-x_\star)\|^{2}_2 \notag
\\
&=  \|   x_{k}-x_\star \|^{2}_2- \| A_{\mathcal{I}_{k}}^{\dagger}A_{\mathcal{I}_{k}} ( x_{k}-x_\star)\|^{2}_2,  \notag
\end{align}
which together with the Courant-Fisher theorem:
\begin{align}
\|Ax\|^2_2\geq\lambda_{\min}\left( A^{T} A\right)\|x\|^2_2  \label{2}\ \textrm{is valid for any vector}\ x \in {\rm R(A^T)},
\end{align}
 and the fact $\lambda_{\min} ((A_{\mathcal{I}_{k}}^{\dagger})^TA_{\mathcal{I}_{k}}^{\dagger})=\lambda_{\max}^{-1}(A_{\mathcal{I}_{k}}^TA_{\mathcal{I}_{k}})$ yields
\begin{align}
\|x_{k+1}-x_\star\|^{2}_2
&\leq  \|   x_{k}-x_\star \|^{2}_2-\lambda_{\min} ((A_{\mathcal{I}_{k}}^{\dagger})^TA_{\mathcal{I}_{k}}^{\dagger})\| A_{\mathcal{I}_{k}} ( x_{k}-x_\star)\|^{2}_2\notag
\\
&= \| x_{k}-x_\star \|^{2}_2-\lambda_{\max}^{-1}(A_{\mathcal{I}_{k}}^TA_{\mathcal{I}_{k}})\| A_{\mathcal{I}_{k}} ( x_{k}-x_\star)\|^{2}_2 \notag
\\
&= \| x_{k}-x_\star \|^{2}_2-\lambda_{\max}^{-1}(A_{\mathcal{I}_{k}}^TA_{\mathcal{I}_{k}})\sum \limits _{i_k\in \mathcal{I}_{k}}(A_{(i_k)} x_{k}- b_{(i_k)})^2. \notag
\end{align}
On the other hand, from Algorithm \ref{alg2}, if $i_k\in \mathcal{I}_{k}$, we have
\begin{align}
( A_{(i_k)} x_{k}- b_{(i_k)})^2
&\geq \delta_k  =\max \limits _{i\in \tau_k} (  A_{(i)} x_{k}-b_{(i)})^2 = \|A_{\tau_k} x_{k}- b_{\tau_k}\|_{\infty}^{2}.  \notag
\end{align}
Then
\begin{align}
\|x_{k+1}-x_\star\|^{2}_2
&\leq \| x_{k}-x_\star \|^{2}_2-\frac{|\mathcal{I}_{k}|}{\lambda_{\max} (A_{\mathcal{I}_{k}}^TA_{\mathcal{I}_{k}})} \|A_{\tau_k} x_{k}- b_{\tau_k}\|_{\infty}^{2}.   \notag
\end{align}
Now, taking expectation of both sides (with respect to the sampled $\tau_k$), we have
\begin{align}
\textrm{E}\|x_{k+1}-x_\star\|^{2}_2
&\leq\| x_{k}-x_\star \|^{2}_2- \textrm{E}\frac{|\mathcal{I}_{k}|}{\lambda_{\max} (A_{\mathcal{I}_{k}}^TA_{\mathcal{I}_{k}})} \|A_{\tau_k} x_{k}- b_{\tau_k}\|_{\infty}^{2} \notag
\\
&= \| x_{k}-x_\star \|^{2}_2- \sum  \limits _{\tau_k \in \binom{m}{\beta}} \frac{1}{\binom{m}{\beta}}\frac{|\mathcal{I}_{k}|}{\lambda_{\max} (A_{\mathcal{I}_{k}}^TA_{\mathcal{I}_{k}})} \|A_{\tau_k} x_{k}- b_{\tau_k}\|_{\infty}^{2}  \notag
\\
&= \| x_{k}-x_\star \|^{2}_2- \frac{1}{\binom{m}{\beta}}\frac{|\mathcal{I}_{k}|}{\lambda_{\max} (A_{\mathcal{I}_{k}}^TA_{\mathcal{I}_{k}})} \sum  \limits _{\tau_k \in \binom{m}{\beta}}  \|A_{\tau_k} x_{k}- b_{\tau_k}\|_{\infty}^{2},  \notag
\end{align}
which together with $\xi_k$ 
defined in \eqref{1000} leads to
\begin{align}
\textrm{E}\|x_{k+1}-x_\star\|^{2}_2
&\leq\| x_{k}-x_\star \|^{2}_2- \frac{1}{\binom{m}{\beta}}\frac{|\mathcal{I}_{k}|}{\lambda_{\max} (A_{\mathcal{I}_{k}}^TA_{\mathcal{I}_{k}})} \frac{1}{\xi_k}\sum  \limits _{\tau_k \in \binom{m}{\beta}}\|A_{\tau_k} x_{k}- b_{\tau_k}\|_{2}^{2}  \notag
\\
&=\| x_{k}-x_\star \|^{2}_2- \frac{1}{\binom{m}{\beta}}\frac{|\mathcal{I}_{k}|}{\lambda_{\max} (A_{\mathcal{I}_{k}}^TA_{\mathcal{I}_{k}})} \frac{1}{\xi_k}\frac{\binom{m}{\beta}\beta}{m}\|A x_{k}- b\|_{2}^{2}.  \notag
\end{align}
Further, considering (\ref{2}), we get
\begin{align}
\textrm{E}\|x_{k+1}-x_\star\|^{2}_2
\leq\left(1- \frac{\beta}{\xi_k}\frac{ |\mathcal{I}_{k}|}{ m}\frac{\lambda_{\min}(A^TA)}{\lambda_{\max} (A_{\mathcal{I}_{k}}^TA_{\mathcal{I}_{k}})}  \right)\| x_{k}-x_\star \|^{2}_2,  \notag
\end{align}
which is the desired result.
\end{proof}

\begin{remark}
\label{rmk3}
According to Lemma \ref{theorem1}, it is easy to see that $\frac{\lambda_{\min}(A^TA)}{\lambda_{\max} (A_{\mathcal{I}_{k}}^TA_{\mathcal{I}_{k}})}\leq1$, which together with the fact $1\leq\xi_k\leq\beta$ yields
$$1- \frac{\beta}{\xi_k}\frac{ |\mathcal{I}_{k}|}{ m}\frac{\lambda_{\min}(A^TA)}{\lambda_{\max} (A_{\mathcal{I}_{k}}^TA_{\mathcal{I}_{k}})}\leq1- \frac{ |\mathcal{I}_{k}|}{ m}\frac{\lambda_{\min}(A^TA)}{\lambda_{\max} (A_{\mathcal{I}_{k}}^TA_{\mathcal{I}_{k}})}<1.  $$
That is, the convergence factor of the BSKM1 method is indeed smaller than 1.

\end{remark}
\begin{remark}
\label{rmk4}
From Algorithm \ref{alg2}, we know that $\{x|A_{\mathcal{I}_{k}}x=b_{\mathcal{I}_{k}}\}\subset\{x|A_{(t_{k})}x=b_{(t_{k})}\}$ since $t_k \in \mathcal{I}_{k}$. Similar to the analysis in \cite{haddock2019greed}, we can obtain $$\|x_k^{SKM}-x_{k-1}\|^2_2\leq\|x_k^{BSKM1}-x_{k-1}\|^2_2,$$
which together with the fact
\begin{align*}
\left\|{x}_{k}^{S K M}-{x}_{k-1}\right\|^{2}_2+\left\|{x}_{k}^{S K M}-{x}_{\star}\right\|^{2}_2&=\left\|{x}_{k-1}-{x}_{\star}\right\|^{2}_2\\
&=\left\|{x}_{k}^{BSKM1}-{x}_{k-1}\right\|^{2}_2+\left\|{x}_{k}^{BSKM1}-{x}_{\star}\right\|^{2}_2 \notag
\end{align*}
leads to
$$\|x_k^{BSKM1}-x_{\star}\|^2_2\leq\|x_k^{SKM}-x_{\star}\|^2_2.$$
In the above expressions, $x_k^{BSKM1}$ and $x_k^{SKM}$ denote the next approximations  generated by the BSKM1 and SKM methods, respectively. Hence, the BSKM1 method converges at least as quickly as the SKM method.

In addition, setting $l_k={\rm arg} \max \limits _{1\leq i\leq m} ( b_{(i)}- A_{(i)} x_{k})^2  $, from Algorithm \ref{alg2}, we can obtain $\{x|A_{\mathcal{I}_{k}}x=b_{\mathcal{I}_{k}}\}\subset\{x|A_{(l_{k})}x=b_{(l_{k})}\}$ since $l_k \in \mathcal{I}_{k}$. So, we also immediately get that the BSKM1 method converges at least as quickly as the Motzkin method.
\end{remark}

\section{The second block sampling Kaczmarz-Motzkin method }\label{sec4}
Considering that $|\mathcal{I}_{k}|$ in Algorithm \ref{alg2} may be $1$ for all $k=1, 2, \ldots$ and 
the size of the index set $\mathcal{I}_{k}$ cannot be controlled,
we design the second block sampling Kaczmarz-Motzkin (BSKM2) method, which is presented in Algorithm \ref{alg3}. 
The biggest difference between the BSKM2 and BSKM1 methods is the way to build the index set. For Algorithm \ref{alg3}, we can control the size of the index set $ \mathcal{J}_{k}$.  

\begin{algorithm}
\caption{ The BSKM2 method} 
\label{alg3}
\begin{algorithmic}
\STATE{\mbox{Input:} ~ Matrix $A\in R^{m\times n}$, vector $b\in R^{m}$, parameter $\eta$, initial estimate $x_0$.}
\STATE{\mbox{Output:} ~Approximate $x$ solving $Ax=b$.}
\STATE{1. For $j=1:\eta$}
\STATE{2. Choose a sample of $\beta_j$ constraints, $\tau_j$, uniformly at random from among the rows of $A$ without replacement.} 
\STATE{3. Compute $t_j={\rm arg} \max \limits _{i\in \tau_j} ( b_{(i)}- A_{(i)} x_{k})^2  $.}
\STATE{4. End for}
\STATE{5. Determine the index set
\begin{align}
  \mathcal{J}_{k}=\{t_1, t_2, \ldots, t_{\eta}\}.\notag
\end{align}}
\STATE{6. Update
\begin{align}
  x_{k+1}=x_{k}+A_{\mathcal{J}_{k}}^{\dagger}(b_{\mathcal{J}_{k}}-A_{\mathcal{J}_{k}}x_{k}).\notag
\end{align}}
\STATE{7. Repeat.}
\end{algorithmic}
\end{algorithm}

\begin{remark}
\label{rmk5}
Note that if
$$( b_{(i_k)}- A_{(i_k)} x_{k})^2 =  \max \limits _{1 \leq i  \leq m}( b_{(i)}- A_{(i)} x_{k})^2 ,$$
then $i_k\in\mathcal{J}_{k}.$ So the index set $\mathcal{J}_{k}$ in Algorithm \ref{alg3} is always nonempty. 
\end{remark}
\begin{remark}
\label{rmk6}
The iteration index $t_k$ used for updating of the SKM method belongs to the index set $\mathcal{J}_{k}$ used in the BSKM2 method. So the latter makes progress faster than the former. In addition, If $\eta=1$, the BSKM2 method reduces to the Motzkin method.
\end{remark}

Now, we bound the expected rate of convergence for Algorithm \ref{alg3}.
\begin{theorem}
\label{theorem3}
From an initial guess $x_0\in{\rm R(A^T)}$, the sequence $\{x_k\}_{k=0}^\infty$ generated by the BSKM2 method converges linearly in expectation to the least-Euclidean-norm solution $x_{\star}=A^{\dag}b$ and
\begin{eqnarray*}
\textrm{E}\|x_{k+1}-x_\star\|^{2}_2\leq\left(1- \frac{\eta \lambda_{\min} (A_{\tau_{h}}^TA_{\tau_{h}})}{|\tau_{h}|\lambda_{\max} (A_{\mathcal{J}_{k}}^TA_{\mathcal{J}_{k}})}\right) \|x_{k}-x_\star \|_{2}^{2},
\end{eqnarray*}
where $\tau_{h}$ satisfies $\|A_{\tau_{h}} x_{k}- b_{\tau_{h}}\|_{\infty}^{2}=\min \limits _{1\leq i\leq \eta} \|A_{\tau_{i}} x_{k}- b_{\tau_{i}}\|_{\infty}^{2}$.
\end{theorem}
\begin{proof}
Following an analogous argument to Theorem \ref{theorem2}, we can obtain
\begin{align}
\|x_{k+1}-x_\star\|^{2}_2
&\leq  \|   x_{k}-x_\star \|^{2}_2-\lambda_{\min} ((A_{\mathcal{J}_{k}}^{\dagger})^TA_{\mathcal{J}_{k}}^{\dagger})\| A_{\mathcal{J}_{k}} ( x_{k}-x_\star)\|^{2}_2\notag
\\
&= \| x_{k}-x_\star \|^{2}_2-\lambda_{\max}^{-1}(A_{\mathcal{J}_{k}}^TA_{\mathcal{J}_{k}})\| A_{\mathcal{J}_{k}} ( x_{k}-x_\star)\|^{2}_2 \notag
\\
&= \| x_{k}-x_\star \|^{2}_2-\lambda_{\max}^{-1}(A_{\mathcal{J}_{k}}^TA_{\mathcal{J}_{k}})\sum \limits _{i_k\in \mathcal{J}_{k}}(A_{(i_k)} x_{k}- b_{(i_k)})^2 \notag
\\
&= \| x_{k}-x_\star \|^{2}_2-\lambda_{\max}^{-1}(A_{\mathcal{J}_{k}}^TA_{\mathcal{J}_{k}})\sum \limits _{j=1}^{\eta}\|A_{\tau_j} x_{k}- b_{\tau_j}\|_{\infty}^{2}. \notag
\end{align}
Now, taking expectation of both sides, we have
\begin{align}
\textrm{E}\|x_{k+1}-x_\star\|^{2}_2
&\leq\| x_{k}-x_\star \|^{2}_2- \textrm{E}\lambda_{\max}^{-1}(A_{\mathcal{J}_{k}}^TA_{\mathcal{J}_{k}})\sum \limits _{j=1}^{\eta}\|A_{\tau_j} x_{k}- b_{\tau_j}\|_{\infty}^{2}\notag
\\
&= \| x_{k}-x_\star \|^{2}_2- \lambda_{\max}^{-1}(A_{\mathcal{J}_{k}}^TA_{\mathcal{J}_{k}})\sum \limits _{j=1}^{\eta}\textrm{E} \|A_{\tau_j} x_{k}- b_{\tau_j}\|_{\infty}^{2}.\notag
\end{align}
Note that
\begin{align}
\sum \limits _{j=1}^{\eta}\textrm{E} \|A_{\tau_j} x_{k}- b_{\tau_j}\|_{\infty}^{2}
&= \textrm{E}\|A_{\tau_1} x_{k}- b_{\tau_1}\|_{\infty}^{2}  +\textrm{E} \|A_{\tau_2} x_{k}- b_{\tau_2}\|_{\infty}^{2}+\cdots \textrm{E}  \|A_{\tau_{\eta}} x_{k}- b_{\tau_{\eta}}\|_{\infty}^{2}\notag
\\
&=  \sum  \limits _{\tau_1 \in \binom{m}{\beta_1}} \frac{1}{\binom{m}{\beta_1}}   \|A_{\tau_1} x_{k}- b_{\tau_1}\|_{\infty}^{2}  +\sum  \limits _{\tau_2 \in \binom{m-\beta_1}{\beta_2}} \frac{1}{\binom{m-\beta_1}{\beta_2}}   \|A_{\tau_2} x_{k}- b_{\tau_2}\|_{\infty}^{2}              \notag
\\
&~~~~+\cdots \sum  \limits _{\tau_{ \eta} \in \binom{m-\beta_1-\ldots-\beta_{\eta-1}}{\beta_\eta}} \frac{1}{\binom{m-\beta_1-\ldots-\beta_{\eta-1}}{\beta_\eta}}   \|A_{\tau_{\eta}} x_{k}- b_{\tau_{\eta}}\|_{\infty}^{2},  \notag
\end{align}
which together with $\|A_{\tau_{h}} x_{k}- b_{\tau_{h}}\|_{\infty}^{2}=\min \limits _{1\leq i\leq \eta} \|A_{\tau_{i}} x_{k}- b_{\tau_{i}}\|_{\infty}^{2}$ leads to
\begin{align}
\sum \limits _{j=1}^{\eta}\textrm{E} \|A_{\tau_j} x_{k}- b_{\tau_j}\|_{\infty}^{2}
&\geq \sum  \limits _{\tau_1 \in \binom{m}{\beta_1}} \frac{1}{\binom{m}{\beta_1}}   \|A_{\tau_h} x_{k}- b_{\tau_h}\|_{\infty}^{2}  +\sum  \limits _{\tau_2 \in \binom{m-\beta_1}{\beta_2}} \frac{1}{\binom{m-\beta_1}{\beta_2}}   \|A_{\tau_h} x_{k}- b_{\tau_h}\|_{\infty}^{2}              \notag
\\
&~~~~+\cdots \sum  \limits _{\tau_{ \eta} \in \binom{m-\beta_1-\ldots-\beta_{\eta-1}}{\beta_\eta}} \frac{1}{\binom{m-\beta_1-\ldots-\beta_{\eta-1}}{\beta_\eta}}   \|A_{\tau_{h}} x_{k}- b_{\tau_{h}}\|_{\infty}^{2}  \notag
\\
&=\eta \|A_{\tau_{h}} x_{k}- b_{\tau_{h}}\|_{\infty}^{2}. \notag
\end{align}
Then, we get
\begin{align}
\textrm{E}\|x_{k+1}-x_\star\|^{2}_2
&\leq \| x_{k}-x_\star \|^{2}_2- \frac{\eta}{\lambda_{\max} (A_{\mathcal{J}_{k}}^TA_{\mathcal{J}_{k}})} \|A_{\tau_{h}} x_{k}- b_{\tau_{h}}\|_{\infty}^{2} \notag
\\
&\leq \| x_{k}-x_\star \|^{2}_2- \frac{\eta }{|\tau_{h}|\lambda_{\max} (A_{\mathcal{J}_{k}}^TA_{\mathcal{J}_{k}})} \|A_{\tau_{h}} x_{k}- b_{\tau_{h}}\|_{2}^{2}, \notag
\end{align}
which together with (\ref{2}) leads to
\begin{align}
\textrm{E}\|x_{k+1}-x_\star\|^{2}_2
\leq (1- \frac{\eta \lambda_{\min} (A_{\tau_{h}}^TA_{\tau_{h}})}{|\tau_{h}|\lambda_{\max} (A_{\mathcal{J}_{k}}^TA_{\mathcal{J}_{k}})}) \|x_{k}-x_\star \|_{2}^{2}, \notag
\end{align}
which is the desired result.
\end{proof}

\begin{remark}
\label{rmk7}
According to Lemma \ref{theorem1}, we have $\frac{\lambda_{\min}(A^TA)}{\lambda_{\max} (A_{\mathcal{J}_{k}}^TA_{\mathcal{J}_{k}})}\leq1$ and $\lambda_{\min}(A^TA) \leq \lambda_{\min} (A_{\tau_{h}}^TA_{\tau_{h}})$, which together with the facts $|\tau_{h}|<m$ and $\eta<m$ yield
$$1- \frac{\eta \lambda_{\min} (A_{\tau_{h}}^TA_{\tau_{h}})}{|\tau_{h}|\lambda_{\max} (A_{\mathcal{J}_{k}}^TA_{\mathcal{J}_{k}})}<1- \frac{\eta \lambda_{\min} (A ^TA )}{m\lambda_{\max} (A_{\mathcal{J}_{k}}^TA_{\mathcal{J}_{k}})}<1.  $$
That is, the convergence factor of the BSKM2 method is indeed smaller than 1.

\end{remark}
\begin{remark}
\label{rmk8}
Note that $\{x|A_{\mathcal{J}_{k}}x=b_{\mathcal{J}_{k}}\}\subset\{x|A_{(t_{k})}x=b_{(t_{k})}\}$, where $t_{k}$ is the iteration index of the SKM method. Thus, similar to the analysis in Remark \ref{rmk4}, we immediately obtain that the BSKM2 method converges at least as quickly as the SKM method. In addition, as $\{x|A_{\mathcal{I}_{k}}x=b_{\mathcal{I}_{k}}\}\subset\{x|A_{(l_{k})}x=b_{(l_{k})}\}$, where $l_k={\rm arg} \max \limits _{1\leq i\leq m} ( b_{(i)}- A_{(i)} x_{k})^2  $, we also get that the BSKM2 method converges at least as quickly as the Motzkin method.
\end{remark}

\begin{remark}
\label{rmk9}
From Algorithms \ref{alg2} and \ref{alg3}, we can find that both the update rules of the two methods need to compute the Moore-Penrose pseudoinverse of the row submatrix $A_{\mathcal{I}_{k}}$ or $A_{\mathcal{J}_{k}}$ in each iteration, which may be expensive. To avoid computing the Moore-Penrose pseudoinverse, we can adopt the following pseudoinverse-free iteration format:
$$x_{k+1}=x_{k}-     \sum_{i \in \mathcal{I}_{k}} w_{i} \frac{A_{(i)} x_{k}-b_{(i)}}{\left\|A_{(i)}\right\|^{2}_2} A_{(i)}^{T},$$ where $ w_{i}$ represents
the weight corresponding to the $i$th row.
See \cite{Necoara2019,Du20202,li2020greedy,moorman2020randomized} for a detailed discussion on this topic.
\end{remark}

\section{Numerical experiments}\label{sec5}
In this section, we mainly compare our two block sampling Kaczmarz-Motzkin methods (BSKM1, BSKM2) and the SKM method in terms of the iteration numbers (denoted as ``Iteration'') and computing time in seconds (denoted as ``CPU time(s)'') using the matrix $A\in R^{m\times n}$ from two sets. One is generated randomly by using the MATLAB function \texttt{randn}, and the other one contains the matrices in Table \ref{tab1} from the University of Florida sparse matrix collection \cite{Davis2011}. To compare these methods more clearly, we set $\eta=\beta$. In addition, for the sparse matrices, the density is defined as follows:
\begin{eqnarray*}
\texttt{density}=\frac{\texttt{number of nonzero of an $m\times n$ matrix}}{\texttt{mn}}.
\end{eqnarray*}

\begin{table}[!htbp]\centering
\begin{small}\scriptsize
\caption{The properties of different sparse matrices.} \label{tab1}
 \begin{tabular}{cccccccccccccccccc}
 \hline
\textbf{name}  &\textbf{ch8-8-b2}     & \textbf{ch7-8-b2} &\textbf{Franz7}&\textbf{ ch7-9-b2}&\textbf{mk12-b2} &\textbf{relat7}   \cr
\hline
 $m \times n$  & $18816 \times 1568 $ & $11760 \times 1176$ & $10164 \times 1740$  &$17640 \times 1512$  &$13860 \times 1485 $  &$ 21924 \times 1045 $    \cr

Full rank       & Yes                 &    Yes               & Yes          & Yes            &Yes                   &No           \cr

Density        & 0.19\%              &    0.26\%            & 0.23\%        & 0.20\%            &0.20\%            &0.36\%            \cr

Condition number &  1.6326e+15         &   1.9439e+15           & 5.5318e+15        &  1.6077e+15             &1.8340e+15               & Inf             \cr
 \hline
\end{tabular}
\end{small}
\end{table}
In all the following specific experiments, we generate the solution vector $x_\star\in R^{n}$ using the MATLAB function \texttt{randn}, and set the vector $b\in R^{m}$ to be $b=Ax_\star$. All experiments start from an initial vector $x_0=0$, and terminate once the \emph{relative solution error }(RES), defined by $$\rm RES=\frac{\left\|x_{k}-A^{\dagger}b\right\|^{2}_2}{\left\|A^{\dagger}b\right\|^{2}_2},$$ satisfies $\rm RES<10^{-6}$, or the number of iteration steps exceeds 200,000.

 \begin{figure}[ht]
 \begin{center}
 \includegraphics [height=4.5cm,width=10cm ]{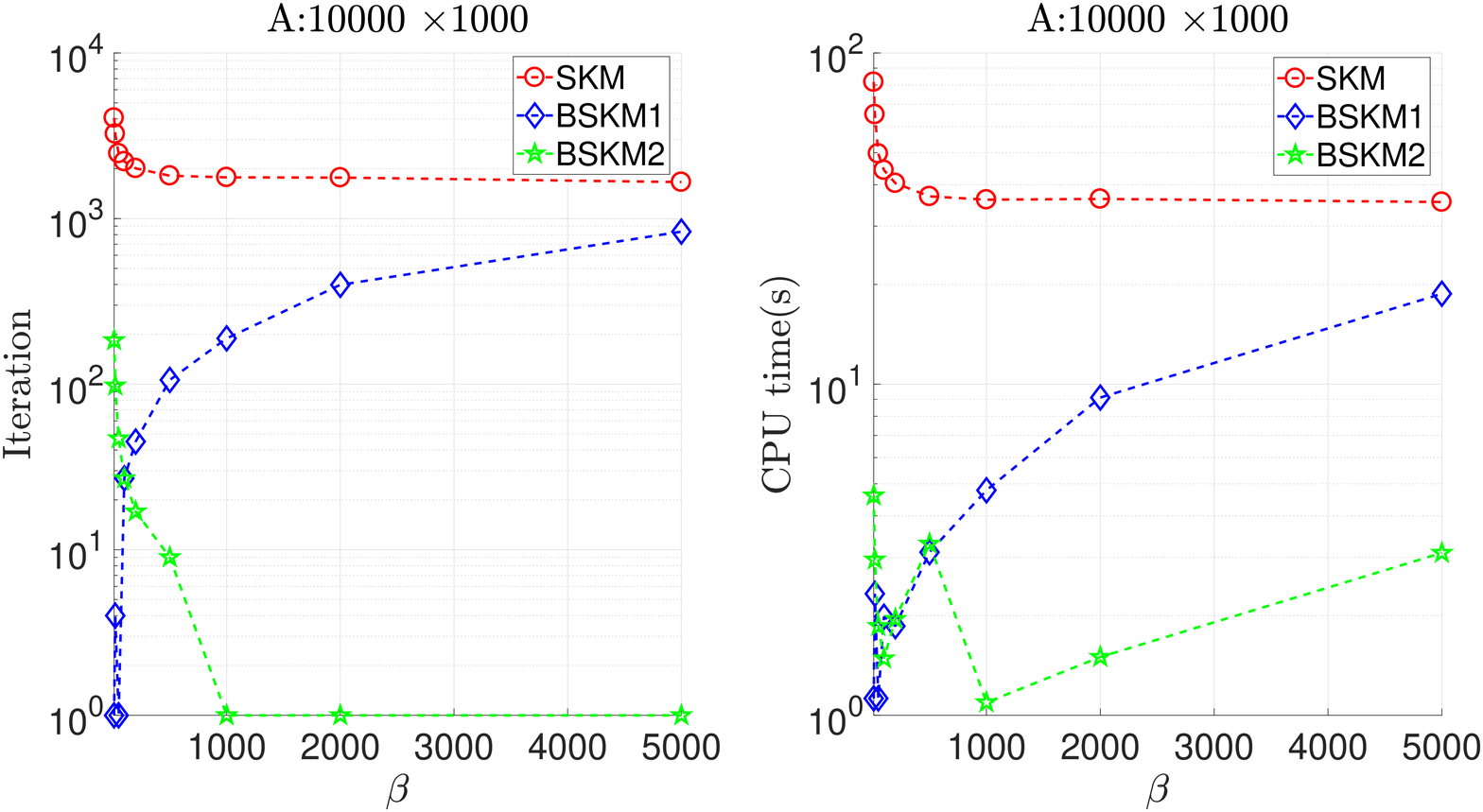}
  \end{center}
\caption{ Iteration and CPU time(s) versus $\beta$ $(10\sim5000)$ for three methods with matrices generated randomly.   }\label{fig1}
\end{figure}
 \begin{figure}[ht]
 \begin{center}
  \includegraphics [height=4.5cm,width=10cm ]{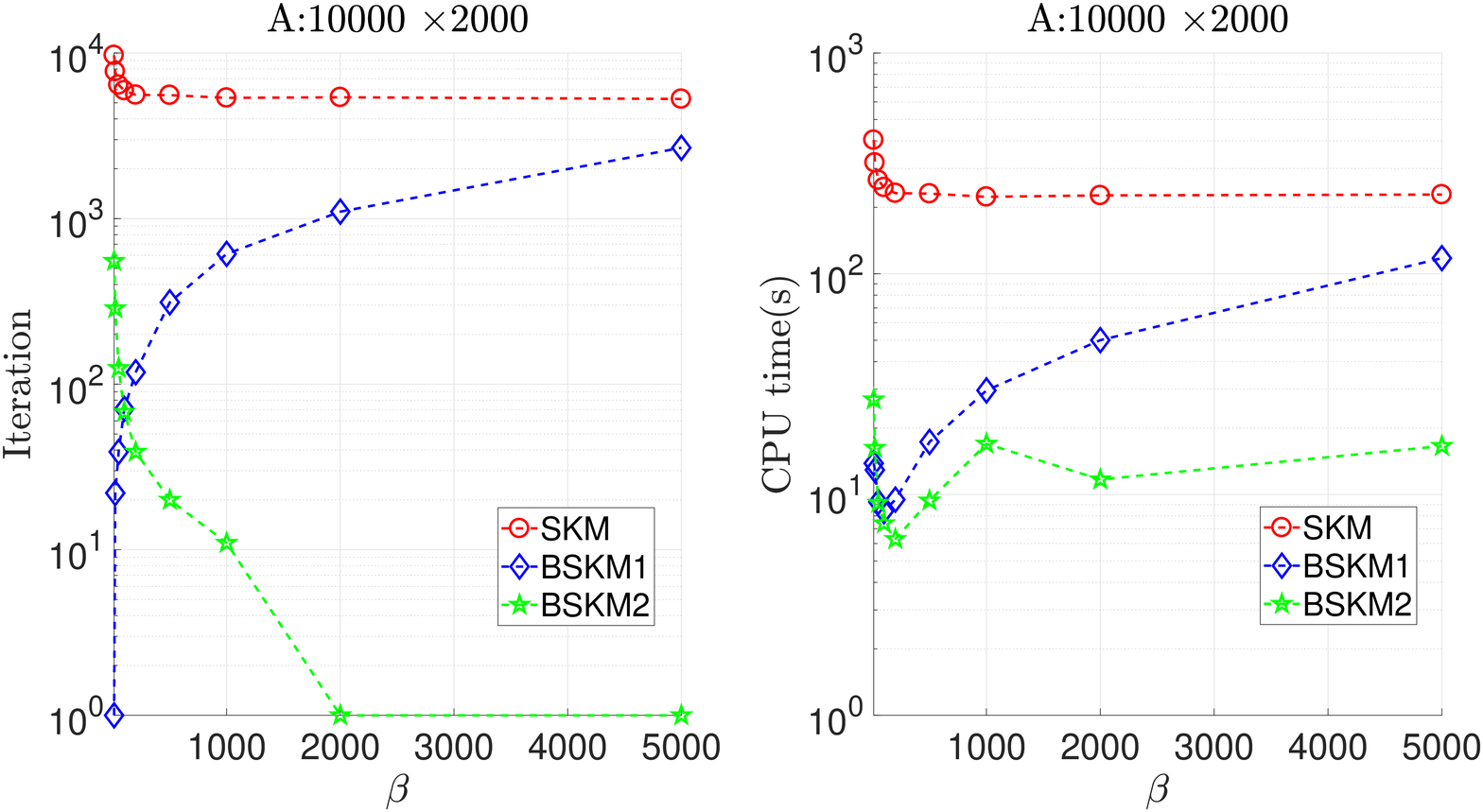}
  \end{center}
\caption{ Iteration and CPU time(s) versus $\beta$ $(10\sim5000)$ for three methods with matrices generated randomly.   }\label{fig2}
\end{figure}
 \begin{figure}[ht]
 \begin{center}
   \includegraphics [height=4.5cm,width=10cm ]{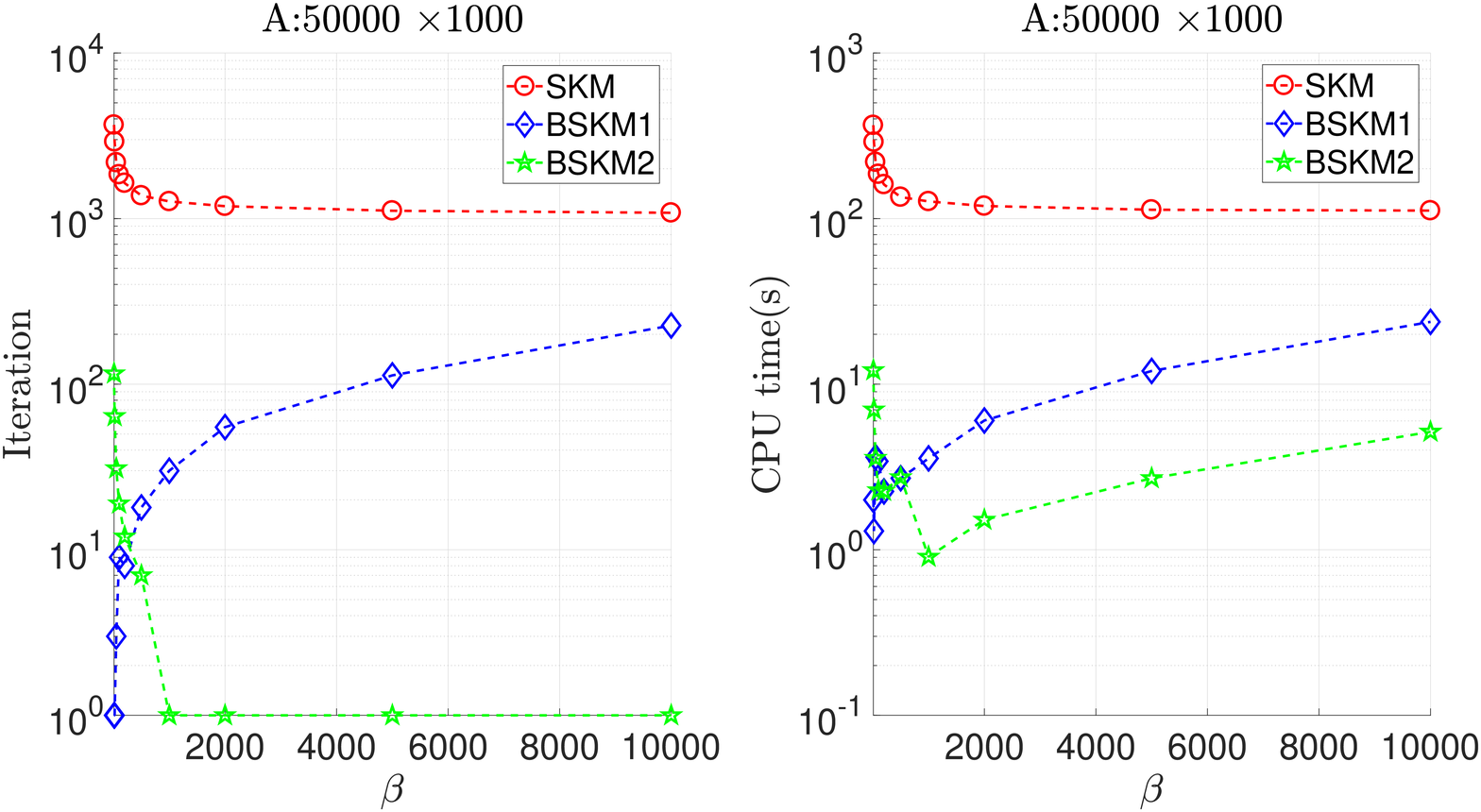}
  \end{center}
\caption{ Iteration and CPU time(s) versus $\beta$ $(10\sim10000)$ for three methods with matrices generated randomly.   }\label{fig3}
\end{figure}
 \begin{figure}[ht]
 \begin{center}
  \includegraphics [height=4.5cm,width=10cm ]{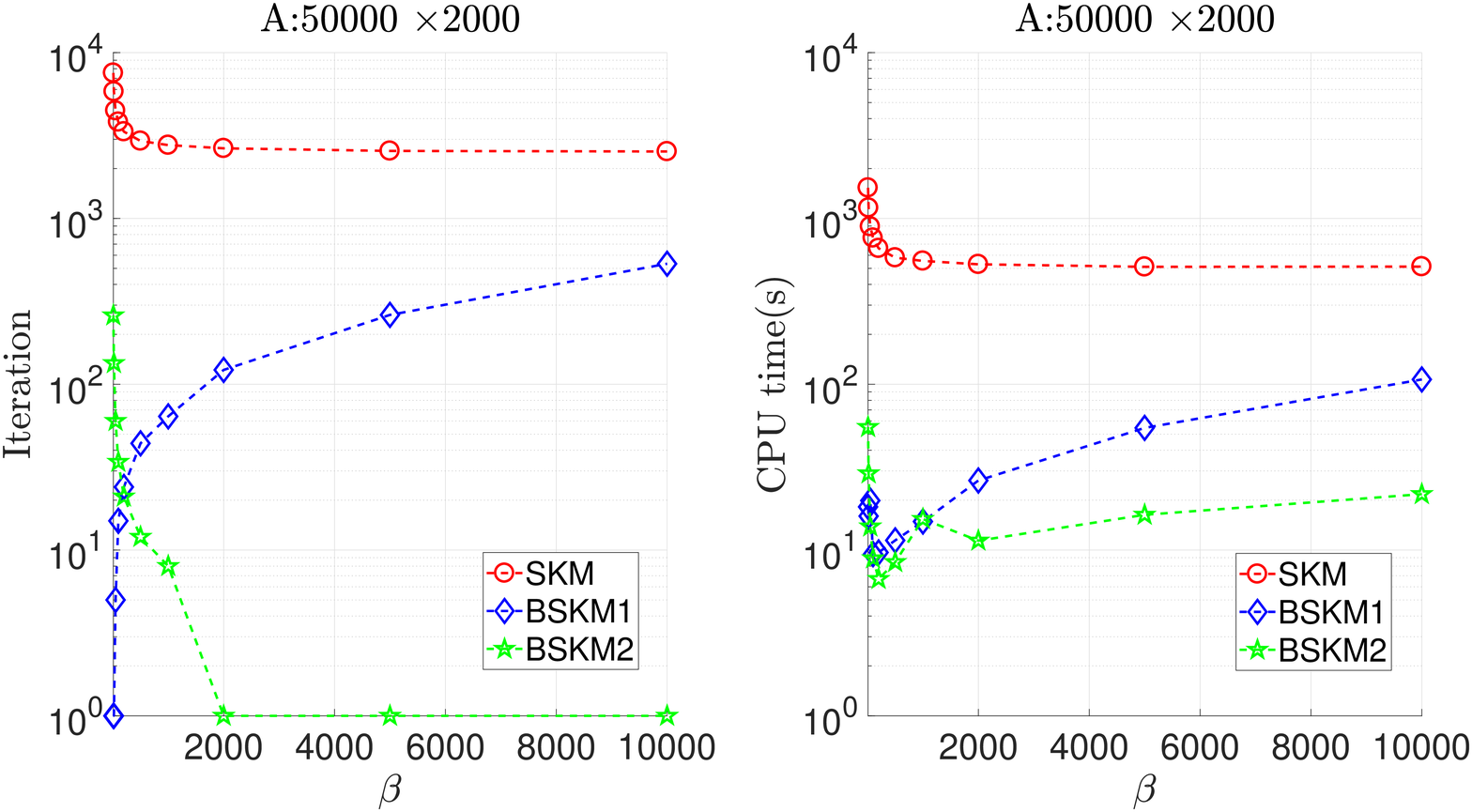}
  \end{center}
\caption{ Iteration and CPU time(s) versus $\beta$ $(10\sim10000)$ for three methods with matrices generated randomly.   }\label{fig4}
\end{figure}
\begin{figure}[ht]
 \begin{center}
\includegraphics [height=4.5cm,width=10cm ]{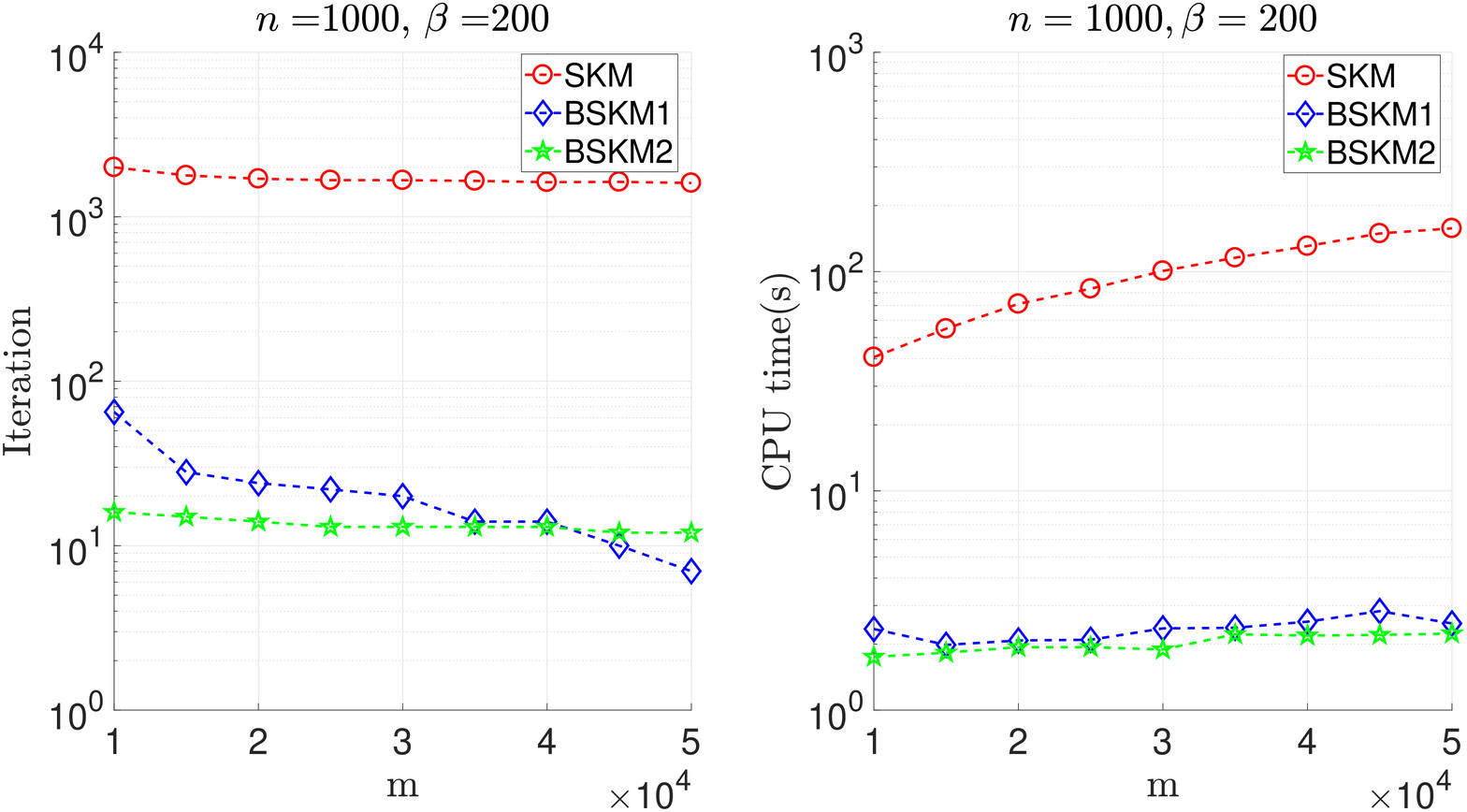}
 \end{center}
\caption{Iteration and CPU time(s) versus m (10000 $\thicksim$ 50000) with matrices generated randomly and $n=1000$ and $\beta=200$.}\label{fig5}
 \end{figure}
\begin{figure}[ht]
 \begin{center}
 \includegraphics [height=4.5cm,width=10cm ]{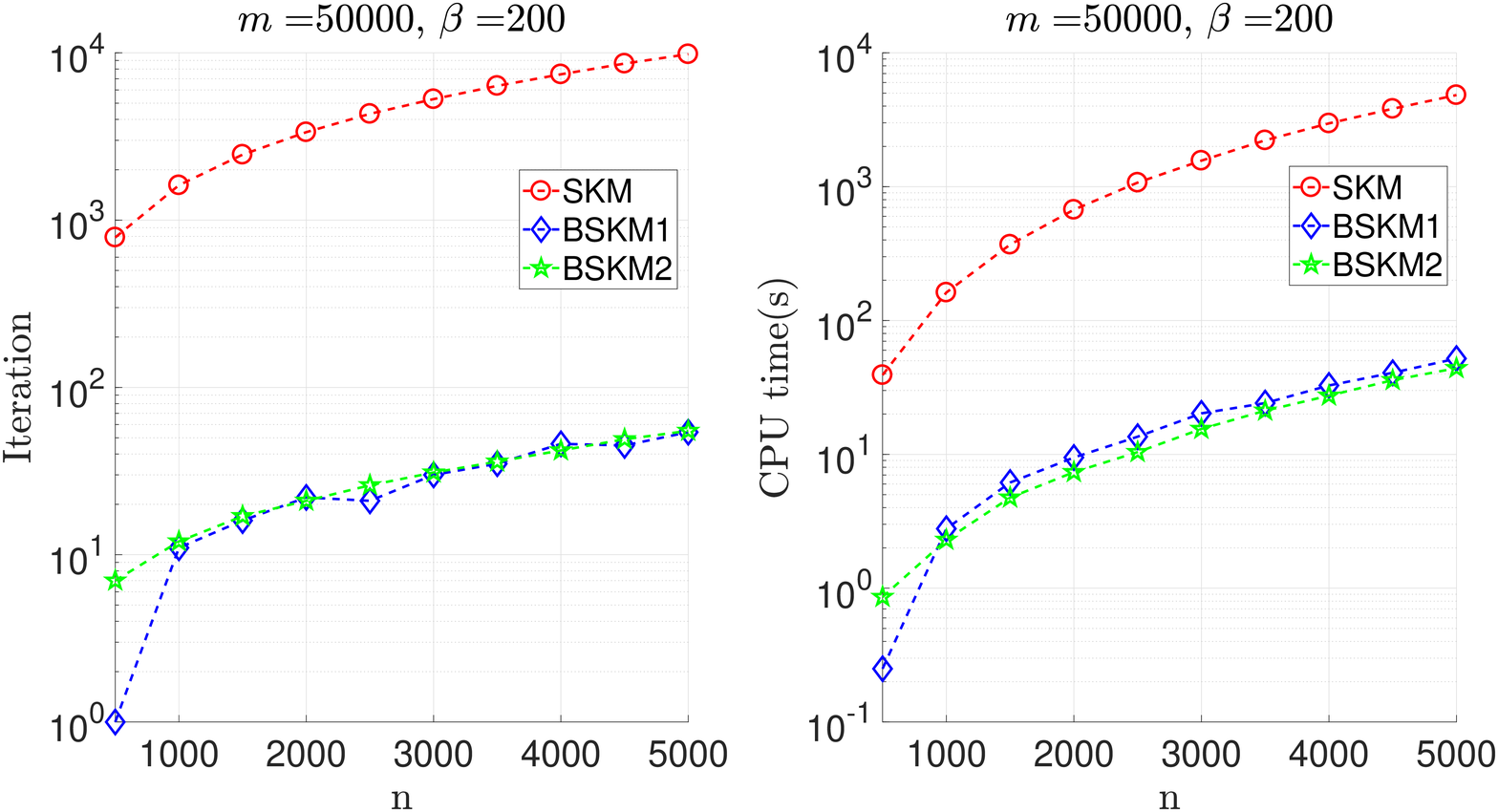}
 \end{center}
\caption{Iteration and CPU time(s) versus n (500 $\thicksim$ 5000) with matrices generated randomly and $m=50000$ and $\beta=200$.}\label{fig6}
 \end{figure}

\begin{figure}[ht]
 \begin{center}
\includegraphics [height=4.5cm,width=10cm ]{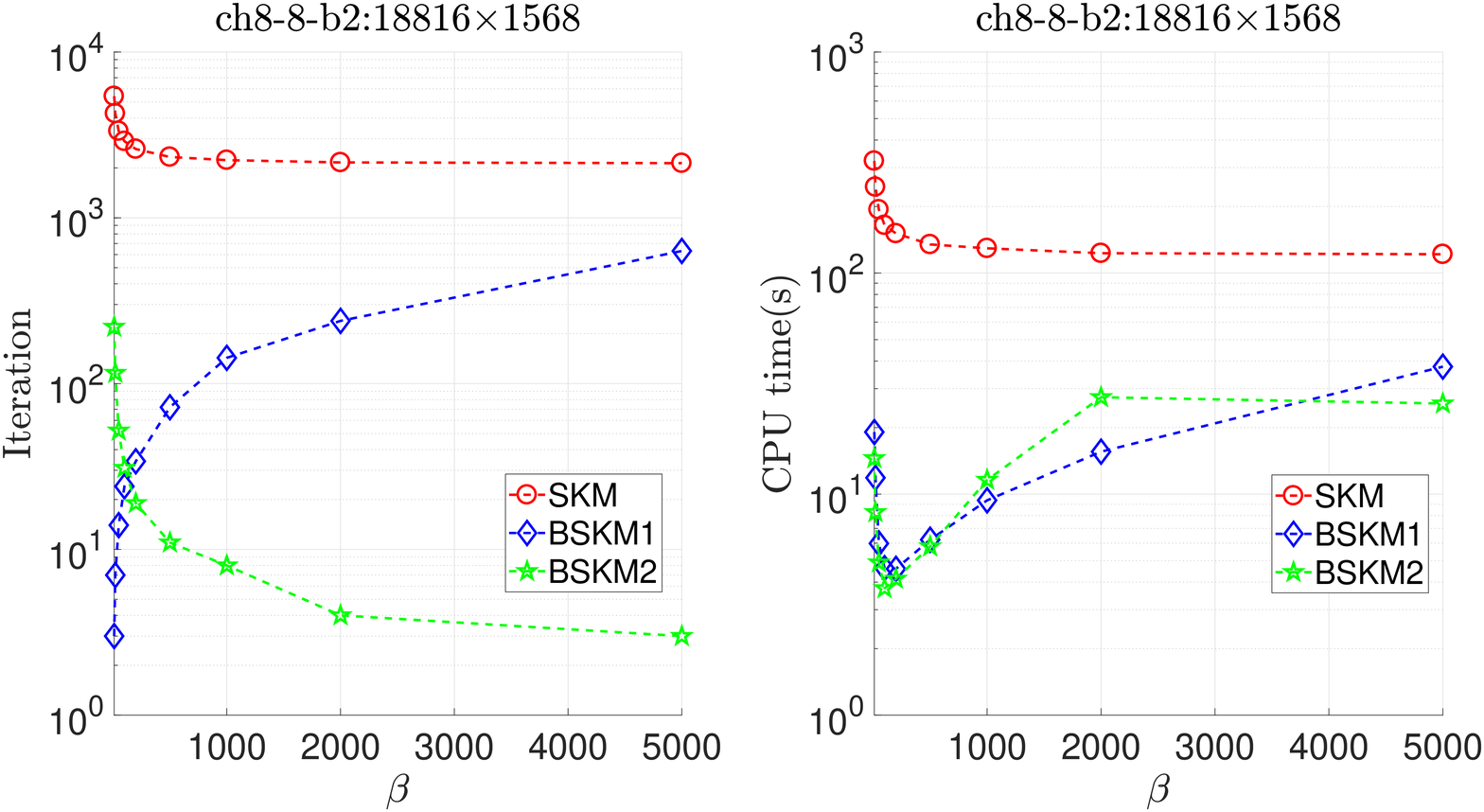}
 \end{center}
\caption{Iteration and CPU time(s) versus $\beta$ $(10\sim5000)$ for three methods with the sparse matrix \textbf{ch8-8-b2}.  }\label{fig7}
\end{figure}

\begin{figure}[ht]
 \begin{center}
 \includegraphics [height=4.5cm,width=10cm ]{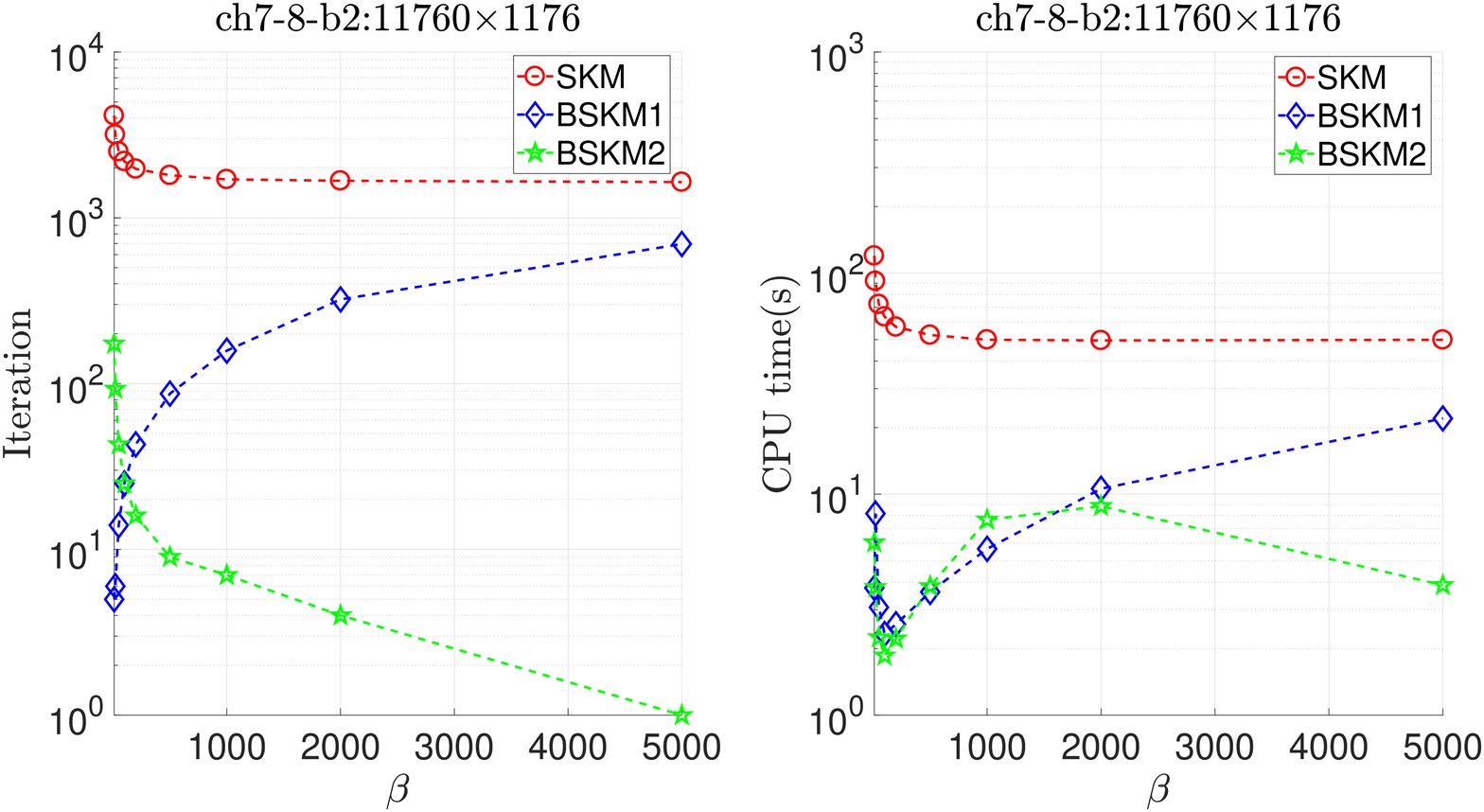}
 \end{center}
\caption{Iteration and CPU time(s) versus $\beta$ $(10\sim5000)$ for three methods with the sparse matrix \textbf{ch7-8-b2}.  }\label{fig8}
\end{figure}

\begin{figure}[ht]
 \begin{center}
 \includegraphics [height=4.5cm,width=10cm ]{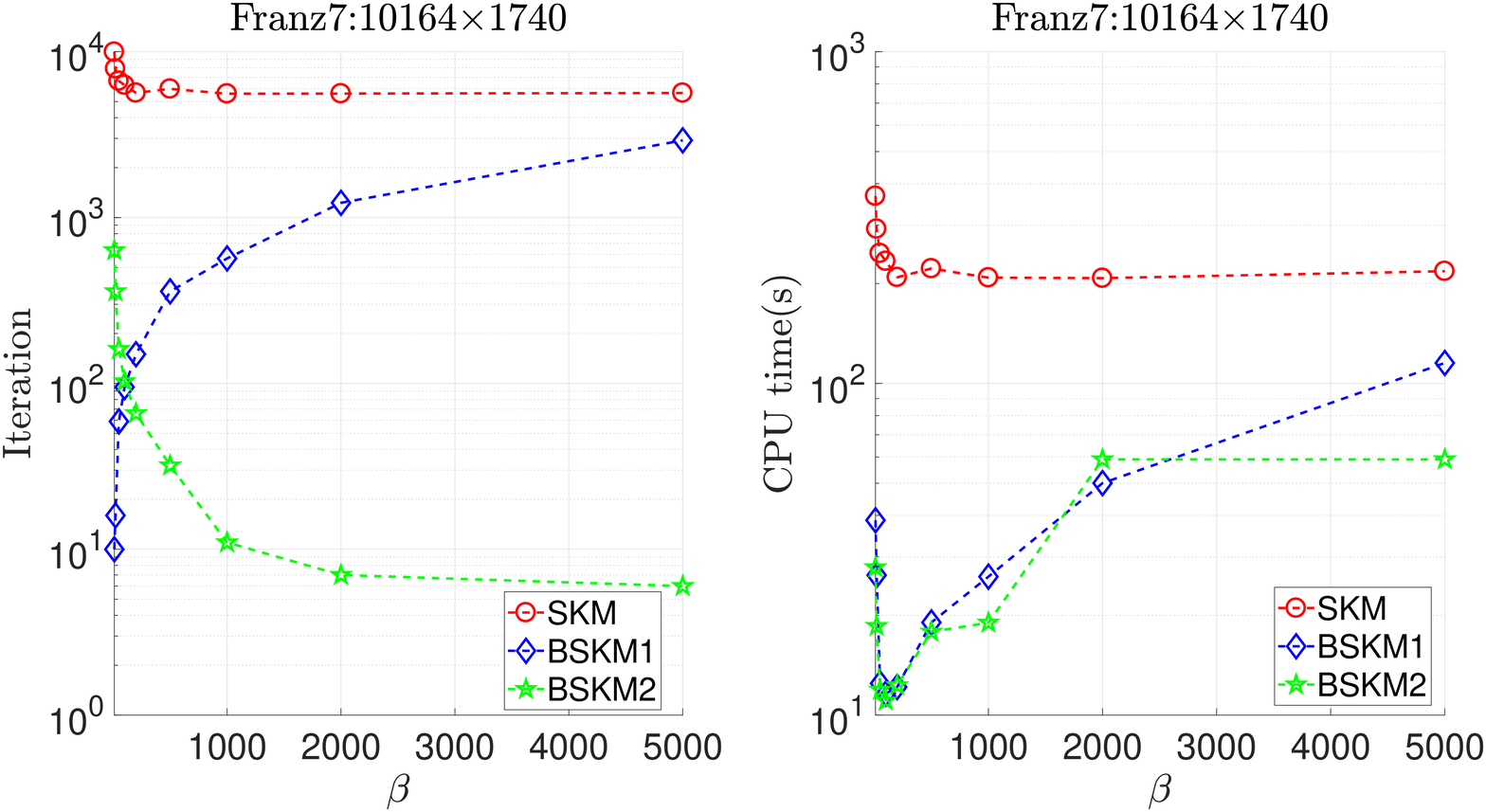}
 \end{center}
\caption{Iteration and CPU time(s) versus $\beta$ $(10\sim5000)$ for three methods with the sparse matrix \textbf{Franz7}.  }\label{fig9}
\end{figure}

\begin{figure}[ht]
 \begin{center}
 \includegraphics [height=4.5cm,width=10cm ]{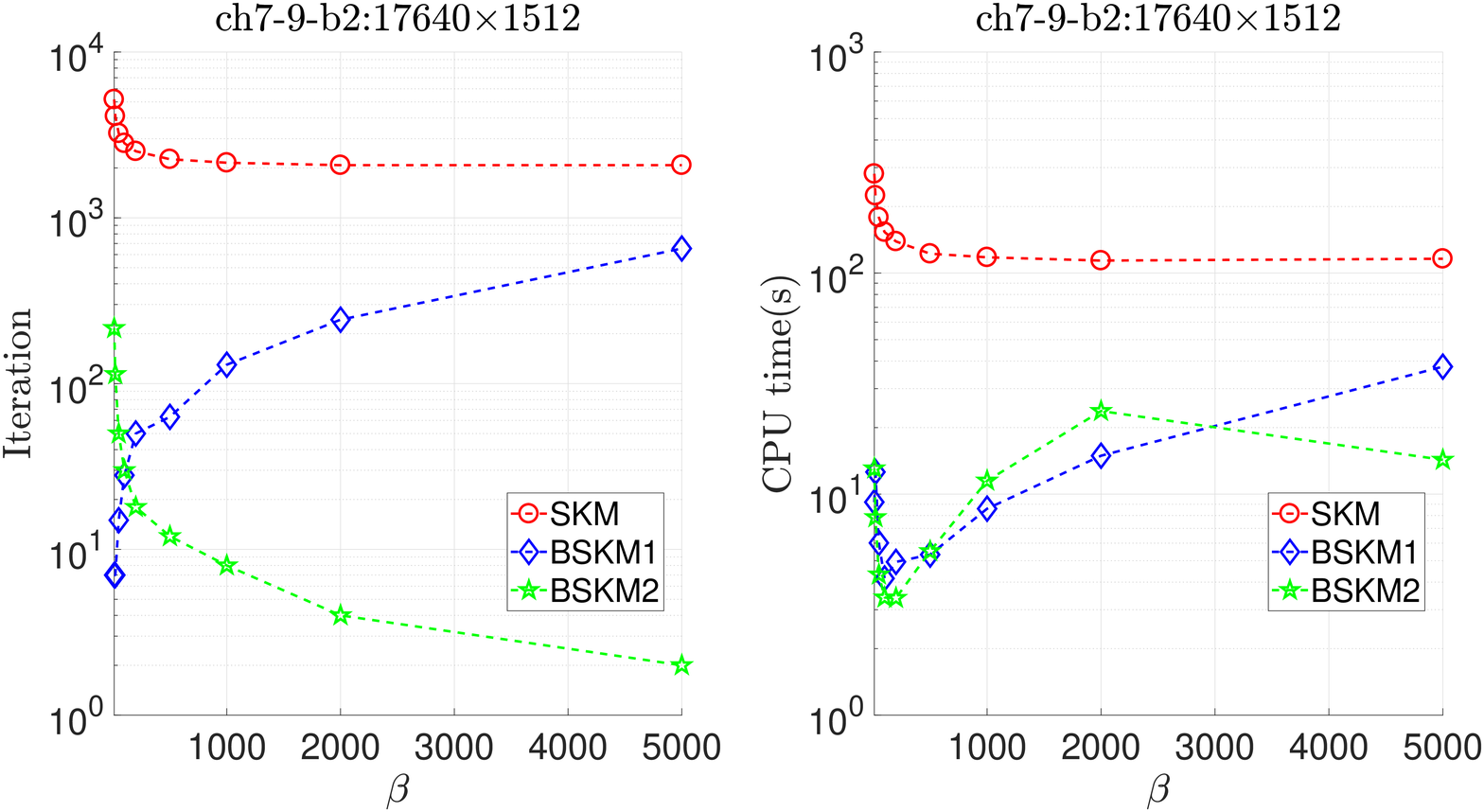}
 \end{center}
\caption{Iteration and CPU time(s) versus $\beta$ $(10\sim5000)$ for three methods with the sparse matrix \textbf{ch7-9-b2}.  }\label{fig10}
\end{figure}

\begin{figure}[ht]
 \begin{center}
 \includegraphics [height=4.5cm,width=10cm]{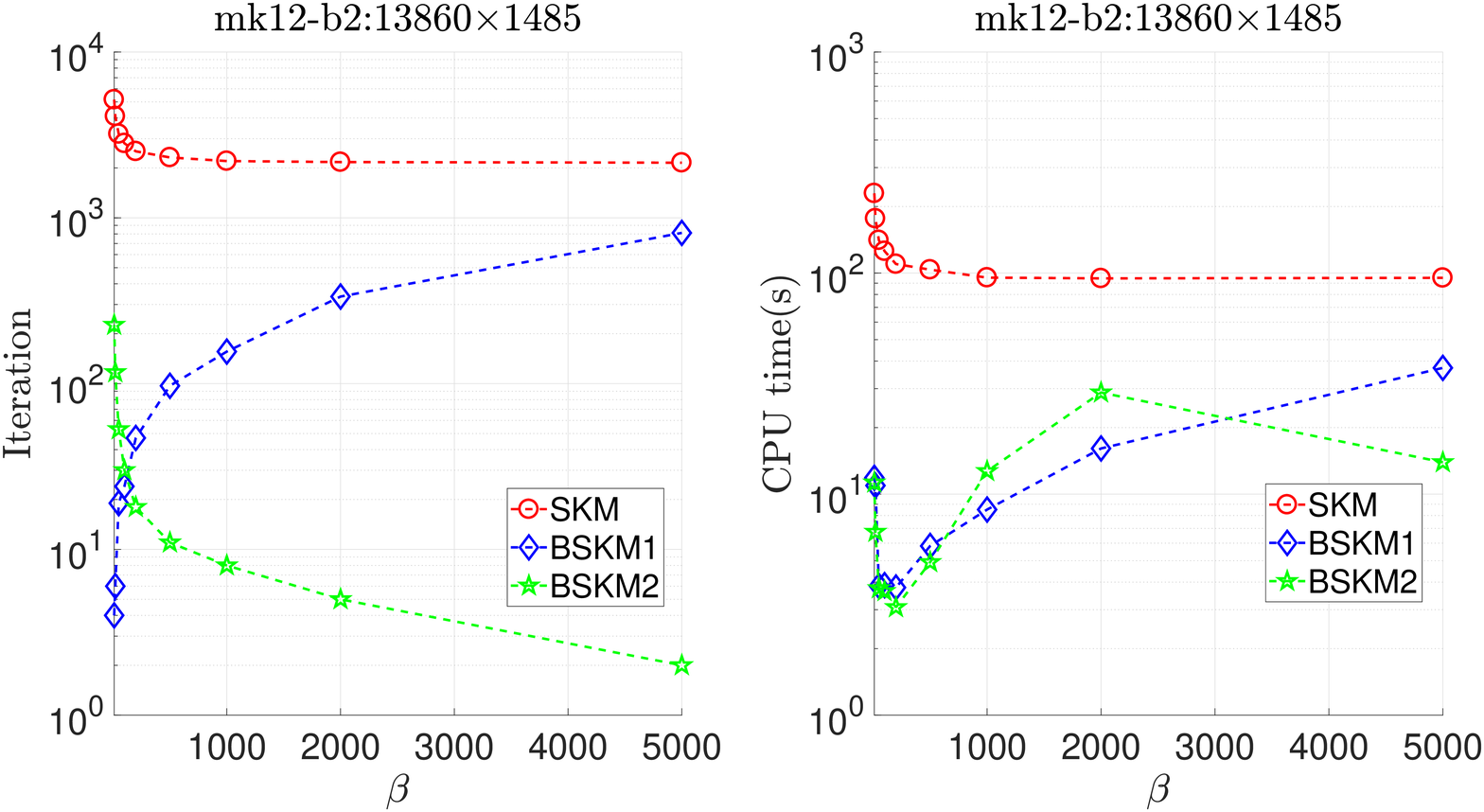}
 \end{center}
\caption{Iteration and CPU time(s) versus $\beta$ $(10\sim5000)$ for three methods with the sparse matrix \textbf{mk12-b2}.  }\label{fig11}
\end{figure}

\begin{figure}[ht]
 \begin{center}
 \includegraphics [height=4.5cm,width=10cm]{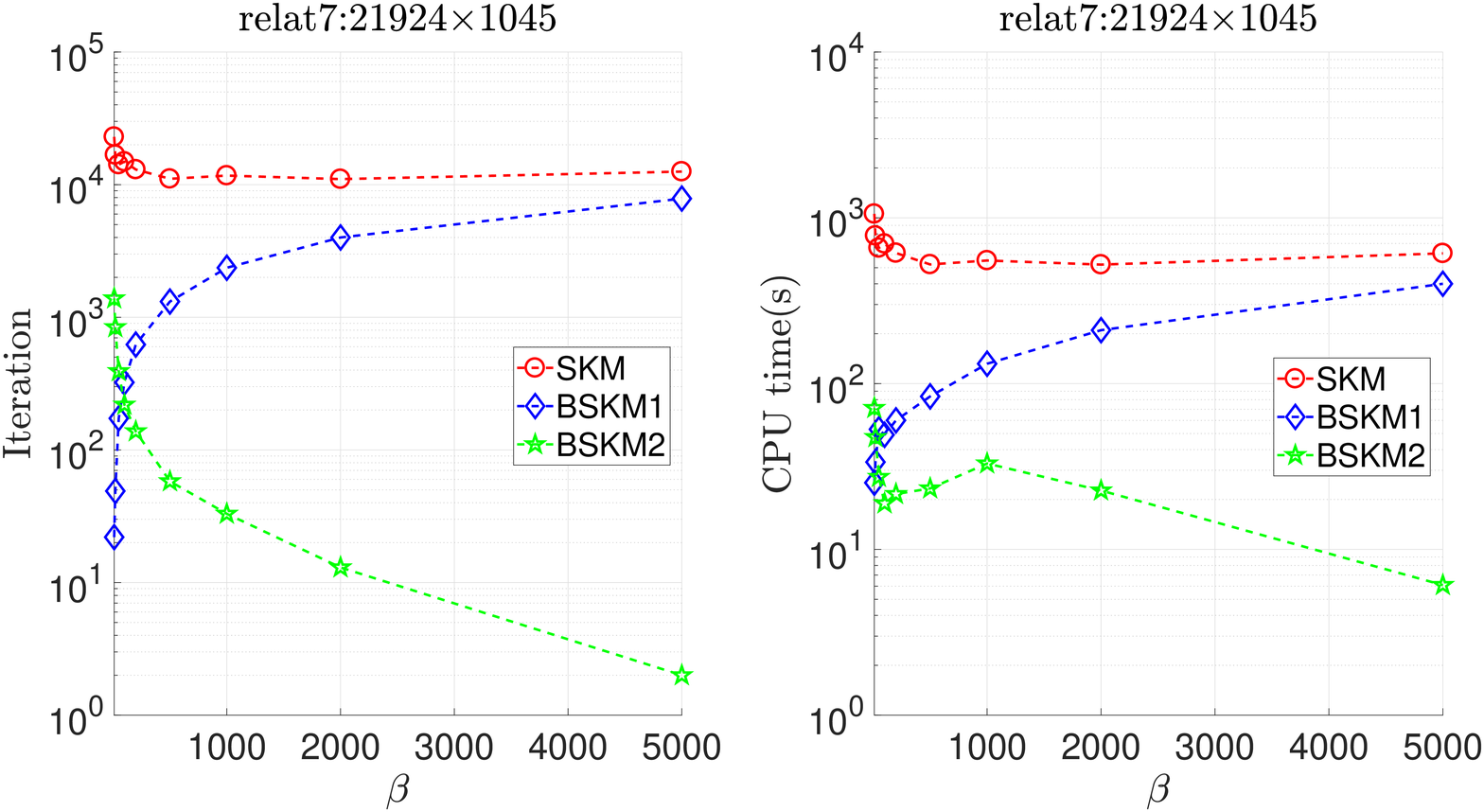}
 \end{center}
\caption{Iteration and CPU time(s) versus $\beta$ $(10\sim5000)$ for three methods with the sparse matrix \textbf{relat7}.  }\label{fig12}
\end{figure}

For the first class of matrices, that is, the matrices generated randomly, the numerical results of the three methods are presented in Figs. \ref{fig1}--\ref{fig6}. Figs. \ref{fig1}--\ref{fig4} show that, with different values of $\beta$, the number of iterative steps and computing time of our two methods are less than those of the SKM method. From Figs. \ref{fig5}--\ref{fig6}, we find that 
the BSKM1 and BSKM2 methods vastly outperform the SKM method in terms of the iterations and computing time when the problems are large-scale.

For the second class of matrices, that is, the sparse matrices from \cite{Davis2011}, we plot the numerical results on Iteration and CPU time(s) versus $\beta$ in Figs. \ref{fig7}--\ref{fig12}. From these figures, we find that the similar results shown in Figs. \ref{fig1}--\ref{fig4}. That is, the BSKM1 and BSKM2 methods converge faster and need less runtime for the same accuracy.

Therefore, in all the cases, our block sampling Kaczmarz-Motzkin methods, i.e., BSKM1 and BSKM2 methods, outperform the SKM method. This is mainly because the latter only updates one index in each iteration while the former 
enforces multiple greedy indices simultaneously.


%
%


\clearpage
\begin{thebibliography}{99}
%
%

\bibitem{kaczmarz1}
S.~Kaczmarz, Angen\"aherte aufl\"osung von systemen linearer gleichungen, Bull.
  Int. Acad. Pol. Sci. Lett. A., 35, 355--357 (1937)

\bibitem{Strohmer2009}
T.~Strohmer, R.~Vershynin, {A randomized Kaczmarz algorithm with exponential
  convergence}, J. Fourier Anal. Appl., 15, 262--278 (2009)

\bibitem{Needell2010}
D.~Needell, {Randomized Kaczmarz solver for noisy linear systems}, BIT Numer.
  Math., 50, 395--403 (2010)

\bibitem{Eldar2011}
Y.~Eldar, D.~Needell, {Acceleration of randomized Kaczmarz method via the
  Johnson-Lindenstrauss lemma}, Numer. Algor., 58, 163--177 (2011)

\bibitem{Completion2013}
A.~Zouzias, M.~N. Freris, {Randomized extended Kaczmarz for solving least
  squares}, SIAM J. Matrix Anal. Appl., 34, 773--793 (2013)

\bibitem{Completion2015}
A.~Ma, D.~Needell, A.~Ramdas, {Convergence properties of the randomized
  extended Gauss-Seidel and Kaczmarz methods}, SIAM J. Matrix Anal. Appl., 36, 1590--1604 (2015),

\bibitem{Dukui2019}
K.~Du, {Tight upper bounds for the convergence of the randomized extended
  Kaczmarz and Gauss-Seidel algorithms}, Numer. Linear Algebra Appl., 26, e2233 (2019)

\bibitem{Wu2020}
N.~C. Wu, H.~Xiang, {Projected randomized Kaczmarz methods}, J. Comput. Appl.
  Math., 372, 112672 (2020)

\bibitem{Chen2020}
J.~Q. Chen, Z.~D. Huang, {On the error estimate of the randomized double block
  Kaczmarz method}, Appl. Math. Comput., 370, 124907 (2020)

\bibitem{Agamon54}
S.~Agamon, The relaxation method for linear inequalities, Canad. J. Math., 6, 382--392 (1954)

\bibitem{Motzkin54}
T.~S. Motzkin, I.~J. Schoenberg, The relaxation method for linear inequalities,
  Canad. J. Math., 6, 393--404 (1954)

\bibitem{Petra2016}
S.~Petra, C.~Popa, {Single projection Kaczmarz extended algorithms}, Numer.
  Algor., 73, 791--806 (2016)

\bibitem{Nutini2016}
J.~Nutini, B.~Sepehry, A.~Virani, I.~Laradji, M.~Schmidt, H.~Koepke,
  {Convergence rates for greedy Kaczmarz algorithms}, presented at UAI (2016)

\bibitem{Nutini2018}
J.~Nutini, {Greed is good: greedy optimization methods for large-scale
  structured problems}, PhD thesis, University of British Columbia (2018)

\bibitem{haddock2019motzkin}
J.~Haddock, D.~Needell, On Motzkin's method for inconsistent linear systems,
  BIT Numer. Math., 59, 387--401 (2019)

\bibitem{rebrova2019sketching}
E.~Rebrova, D.~Needell, Sketching for Motzkin's iterative method for linear
  systems, Proc. 50th Asilomar Conf. on Signals, Systems and Computers (2019)

\bibitem{li2020novel}
H.~Y. Li, Y.~J. Zhang, A novel greedy Kaczmarz method for solving consistent
  linear systems, arXiv preprint arXiv:2004.02062 (2020)

\bibitem{de2017sampling}
J.~A. De~Loera, J.~Haddock, D.~Needell, A sampling Kaczmarz-Motzkin algorithm
  for linear feasibility, SIAM J. Sci. Comput., 39, S66--S87 (2017)

\bibitem{morshed2020accelerated}
M.~S. Morshed, M.~S. Islam, M.~Noor-E-Alam, Accelerated sampling Kaczmarz
  Motzkin algorithm for the linear feasibility problem, J. Global Optim., 77, 361--382 (2020)

\bibitem{haddock2019greed}
J.~Haddock, A.~Ma, Greed works: an improved analysis of sampling
  Kaczmarz-Motzkin, arXiv preprint arXiv:1912.03544 (2019)

\bibitem{needell2014paved}
D.~Needell, J.~A. Tropp, Paved with good intentions: analysis of a randomized
  block Kaczmarz method, Linear Algebra Appl., 441, 199--221 (2014)

\bibitem{needell2015randomized}
D.~Needell, R.~Zhao, A.~Zouzias, Randomized block Kaczmarz method with
  projection for solving least squares, Linear Algebra Appl., 484, 322--343 (2015)

\bibitem{Niu2020}
Y.~Q. Niu, B.~Zheng, {A greedy block Kaczmarz algorithm for solving
  large-scale linear systems}, Appl. Math. Lett., 104, 106294 (2020)

\bibitem{horn2012matrix}
R.~A. Horn, C.~R. Johnson, Matrix analysis, Cambridge Univ. Press (2012)

\bibitem{Necoara2019}
I.~Necoara, {Faster randomized block Kaczmarz algorithms}, SIAM J. Matrix Anal.
  Appl., 40, 1425--1452 (2019)

\bibitem{Du20202}
K.~Du, W.~T. Si, X.~H. Sun, Pseudoinverse-free randomized extended block
  Kaczmarz for solving least squares, arXiv preprint arXiv:2001.04179 (2020)

\bibitem{li2020greedy}
H.~Y. Li, Y.~J. Zhang, Greedy block Gauss-Seidel methods for solving large
  linear least squares problem, arXiv preprint arXiv:2004.02476 (2020)

\bibitem{moorman2020randomized}
J.~D. Moorman, T.~K. Tu, D.~Molitor, D.~Needell, Randomized Kaczmarz with
  averaging, BIT Numer. Math. (2020)

\bibitem{Davis2011}
T.~A. Davis, Y.~F. Hu, {The university of florida sparse matrix collection},
  ACM. Trans. Math. Softw., 38, 1--25 (2011)




\end{thebibliography}


\end{document}